\documentclass[letterpaper,12pt]{amsart}

\title{Strongly Robust Toric Ideals in Codimension 2}

\author{Seth Sullivant}
\address{Department of Mathematics \\ North Carolina State University, Raleigh, NC 27695}
\email{smsulli2@ncsu.edu}

\usepackage{latexsym,array,delarray,amsthm,amssymb,epsfig,hyperref}
\usepackage{tikz}
\usepackage{verbatim}
\usepackage{graphicx}

\theoremstyle{plain}
\newtheorem{thm}{Theorem}[section]
\newtheorem{lemma}[thm]{Lemma}
\newtheorem{prop}[thm]{Proposition}
\newtheorem{cor}[thm]{Corollary}

\newtheorem*{thm*}{Theorem}
\newtheorem*{lemma*}{Lemma}
\newtheorem*{prop*}{Proposition}
\newtheorem*{cor*}{Corollary}
\newtheorem*{conj*}{Conjecture}

\theoremstyle{definition}
\newtheorem{defn}[thm]{Definition}
\newtheorem*{defn*}{Definition}
\newtheorem{ex}[thm]{Example}
\newtheorem{pr}[thm]{Problem}

\newtheorem{ques}[thm]{Question}

\theoremstyle{remark}

\newcommand{\zz}{\mathbb{Z}}
\newcommand{\nn}{\mathbb{N}}
\newcommand{\pp}{\mathbb{P}}
\newcommand{\qq}{\mathbb{Q}}

\newcommand{\kk}{\mathbb{K}}

\newcommand{\calf}{\mathcal{F}}
\newcommand{\calg}{\mathcal{G}}
\newcommand{\calh}{\mathcal{H}}

\newcommand{\cals}{\mathcal{S}}

\newcommand{\conv}{\mathrm{conv}}
\newcommand{\ind}{\mbox{$\perp \kern-5.5pt \perp$}}

\newcommand{\rank}{\textnormal{rank}}

\newcommand{\supp}{\textnormal{supp}}

\tikzstyle{vertex}=[circle, draw, inner sep=0pt, minimum size=6pt, fill=black]

\begin{document}

\begin{abstract}
A homogeneous ideal is robust if its universal Gr\"obner basis is also a 
minimal generating set.  For toric ideals, one has the stronger definition: A toric ideal is strongly robust
 if its
Graver basis equals the set of indispensable binomials. 
 We characterize the codimension 2  strongly robust toric
ideals by their Gale diagrams.  This give a positive answer to a question
of Petrovic, Thoma, and Vladoiu in the case of codimension 2 toric ideals.
\end{abstract}

\maketitle


\section{Introduction}\label{introduction}

A homogeneous ideal is robust if its universal Gr\"obner basis is also a 
minimal generating set.  Although one typically expects the universal 
Gr\"obner basis to be much larger than a minimal generating set (and hence
most ideals are far from robust), there
are a surprising number of examples of ideals that are robust.
Usually these examples have rich underlying combinatorics.
Three well-known example are: the ideals of maximal minors of generic matrices
of indeterminates
\cite{BernsteinZelevinsky1993,SturmfelsZelevinsky1993},  the vanishing ideal of the
closure of an affine linear space in $(\pp^1)^n$ \cite{Ardila2016}, and  toric ideals
of Lawrence type (see \cite[Chapter 7]{Sturmfels1996}).

Let $A \in \zz^{d \times n}$ be an integer matrix of rank $d$, 
and $\kk[p] := \kk[p_1, \ldots, p_n]$
the polynomial ring in $n$ indeterminates.  The toric ideal associated to the
matrix $A$ is the binomial ideal
\[
I_A =  \langle  p^u  -  p^v  :  u,v \in \nn^n,  Au  = Av  \rangle.
\] 
Properties of the generating set of $I_A$ and the geometry of the corresponding
variety are determined by combinatorial properties of the matrix $A$, and many
conditions can be expressed in terms of linear algebra over the integers.
Boocher and Robeva \cite{Boocher2015} initiated a systematic study
of robustness of toric ideals and introduced the word ``robust''.   
They showed that a set of quadratic binomials generate a robust ideal
if and only if it is the direct sum of ideals of maximal minors of  $2 \times n_i$ generic matrices
on disjoint sets of variables.  Since these ideals are toric ideals of Lawrence
type, one wonders if all robust toric ideals must be of Lawrence type.
Petrovic, Thoma, and Vladoiu \cite{Petrovic2015} studied this problem
by introducing an oriented matroid concept they call ``bouquets'', which we
explain below.  They also introduced a strengthening of robust
for toric ideals, which they called $\emptyset$-Lawrence, and we
call strongly robust, that involves looking at a superset of the 
universal Gr\"obner basis called the Graver basis (explained in Section \ref{sec:2}).

One motivation for studying strongly robust toric ideals comes from algebraic
statistics.  Recall that the generating set of a toric ideal is called a Markov
basis.  This is because the binomial generators can be used as a set of moves
to perform a random walk on the fiber $\calf(u) = \{ v \in \nn^n : Au = Av \}$ (see
\cite{Diaconis1998}).  
While any binomial generating set of the toric ideal can be used to generate the associated 
Markov chain, Markov bases that make rapid connections between elements of all fibers should
be preferred since our intuition tells us that these Markov chains will mix more rapidly.  
One desirable property of a Markov basis that guarantees short connections it the
distance-reducing property \cite{Takemura2005}.  Since Graver bases always satisfy the
distance-reducing property, strongly robust toric ideals have the pleasing property that
every Markov basis is distance-reducing.  This suggests  that strongly robust toric ideals
should have nice properties from the standpoint of mixing times of the associated
Markov chain.

Associated to the matrix $A$ is the Gale transform $B$ which is a
$n \times {n-d}$ integer matrix whose columns span $\ker_\zz A$.  
When describing the matrix $A$, we often think about it as a list of
column vectors $A = \{a_1, a_2, \ldots, a_n \}$.  
When describing the Gale transform we think about it as a list of row
vectors $ B = \{b_1, b_2, \ldots, b_n \}$.  A \emph{bouquet} is a maximal subset $S \subseteq [n]$
such that ${\rm span}(b_s : s \in S)$ in one-dimensional.  
A bouquet $S$ is \emph{mixed} if not all elements $\{b_s : s \in S\}$
lie in the same orthant. In the language of matroid theory, 
bouquets correspond to rank one flats of the dual matroid associated to $A$.

A key observation of \cite{Petrovic2015} is that the toric ideals of Lawrence type have many
mixed bouquets.  Recall that if $A \in \zz^{d \times n}$, the \emph{Lawrence lifting}
of $A$ is the matrix
\[
\Lambda(A)  =  \begin{pmatrix}
A  &  0  \\
I  & I  \end{pmatrix}   \in \zz^{(d + n) \times 2n }
\]
where $I$ denotes an $n \times n$ identity matrix.
A toric ideal $I_C$ is said to be of Lawrence type if it is equal to $I_{\Lambda(A)}$
for some matrix $A$, perhaps after permuting the indeterminates.  
Note that 
\[
\ker_\zz  \Lambda(A)  =   \{  (u, -u)  \in \zz^{2n}  :  u \in \ker_\zz A \}.
\]
This means that for a toric ideal of Lawrence type, every $s \in [2n]$ belongs to
a mixed bouquet.   Petrovic, Thoma, and Vladoiu also show how to use the bouquet
structure to produce new examples of strongly robust toric ideals that are not of
Lawrence type, and they posed the following question about strongly robust toric ideals.

\begin{ques}\label{ques:big}
If $I_A$ is a strongly robust toric ideal, must $A$ have a mixed bouquet?
\end{ques}

If $A \in \zz^{d \times n}$ is a toric ideal, with $d = \rank A$, then
the codimension of $I_A$ is $n - d$.  When the codimension of
$I_A$ is one, in which case $I_A$ is a principal ideal,
Question \ref{ques:big} is trivial since $A$ consists of a single bouquet
that must be mixed if $I_A$ is positively graded.
We also provide a positive answer to Question \ref{ques:big}
in the case that $I_A$ has codimension $2$ by giving a complete
characterization of the strongly robust codimension $2$ toric ideals
in terms of the Gale transform,
which is described in following sections.  One consequence is the following:

\begin{thm}\label{thm:main}
Let $A \in \zz^{(n-2) \times n}$ be a full rank matrix, and $\tilde{B} = \{ b_1, \ldots, b_n\} \subseteq \zz^2$
be the reduced Gale transform of $A$. If $I_A$ is a strongly robust toric ideal
then $\conv(\tilde{B})$ is a centrally symmetric polygon.
\end{thm}

In fact, Theorem \ref{thm:main} provides a stronger answer to Question \ref{ques:big}
in the case of codimension $2$ toric ideals.

\begin{cor}\label{cor:2bouquet}
If a codimension $2$ toric ideal $I_A$ is strongly robust then $A$ has at least $2$  mixed bouquets.
\end{cor}

Both of these results will be a consequence of the general characterization of strongly robust
codimension $2$ toric ideals that we prove in the next section.
The proof uses the Peeva-Sturmfels  \cite{Peeva1998} theory of toric ideals of codimension $2$.
While the result of Theorem \ref{thm:main} does not directly generalize to
toric ideals of higher codimension, it does suggest that the property of being strongly robust
is connected to the geometry of the Gale transform, which might suggest other
approaches to Question \ref{ques:big}.

  
\section{Proof of Theorem \ref{thm:main} and Corollary \ref{cor:2bouquet}}\label{sec:2}

To prove Theorem \ref{thm:main} we need more details about strongly robust toric ideals,
and results about generating sets of codimension $2$ toric ideals.  Note that the main definitions and
constructions will be illustrated in Example \ref{ex:Gale}.

First of all, we need to formally introduce the definition of strongly
robust toric ideal \cite{Boocher2015b, Petrovic2015}.
To explain this we introduce some definitions.
Given a vector $u \in \zz^n$ the support of $u$, ${\rm supp}(u) \subseteq [n]$ is the
set of indices $i$ where $u_i$ is not zero.
Let $u \in \nn^n$.  The fiber of $u$ is the set $\calf(u) = \{ v \in \nn^n :  Au = Av \}$.
Clearly if $p^u - p^v \in I_A$ then $u,v$ belong to the same fiber. 
A binomial $p^u -p^v$ is called an \emph{indispensable binomial} if
$\calf(u) = \{u,v\}$ and $\supp(u) \cap \supp(v) = \emptyset$.
The set of all indispensable binomials is denoted $\cals(A)$. 
A binomial $p^u - p^v \in I_A$ is called \emph{primitive} if
there is no other binomial $p^{u'} - p^{v'} \in I_A$ such that
$p^{u'} | p^u$ and $p^{v'} | p^v$.  The set of all primitive binomials in
$I_A$ is called the \emph{Graver basis} of $A$, and denoted
$\calg r(A)$.   The universal Gr\"obner basis of $A$ is a subset of
the Graver basis, and the set of indispensable binomials are a subset
of the universal Gr\"obner basis.  The set of indispensable binomials
appear in every binomial minimal generating set of $I_A$.  This leads
to the following definition:

\begin{defn}
The toric ideal $I_A$ is \emph{strongly robust} if $\cals(A) = \calg r(A)$.
\end{defn}

In \cite{Petrovic2015} strongly robust toric ideals are called $\emptyset$-Lawrence.
Clearly every strongly robust toric ideal is robust.  Boocher et al \cite{Boocher2015b}
wonder (Question 6.1) if robust implies strongly robust for toric ideals, and prove
this is true in some instance associated to graphs.  

One useful tool for analyzing the Graver bases of $I_A$, is its connection to the
Lawrence lifting.  Recall that the definition of the Lawrence lifting from the
introduction.  Its toric ideal $I_{\Lambda(A)}$ we consider to be in the ring
$\kk[p,q]$ with $2n$ indeterminates.  Binomials in $I_{\Lambda(A)}$ have the form
$p^u q^v - p^v q^u$ such that $p^u - p^v \in I_A$.

\begin{thm}\label{thm:Lawrence} \cite[Alg 7.2]{Sturmfels1996} 
Let $A \in \zz^{d \times n}$.  Let $M = \{ p^{u_i} q^{v_i} - p^{v_i} q^{u_i} : i = 1, \ldots, m \}$
 be a binomial minimal generating set of the
toric ideal $I_{\Lambda(A)}$.  Then $\{ p^{u_i}  - p^{v_i}  : i = 1, \ldots, m \}$ is the
Graver basis of $I_A$.
\end{thm}

A key tool for studying toric ideals in codimension $2$ are the
reduced Gale diagrams.  These were used by Peeva and Sturmfels \cite{Peeva1998}
to give a compete description of the free resolution of codimension
$2$ toric ideals.  We define
them now:

Let $A \in \zz^{(n-2) \times n}$ be a matrix of rank $n-2$ and
$B$ the resulting Gale configuration.  Let $B = \{b_1, \ldots, b_n\}$
be the resulting list of row vectors, with $b_i = (b_{i1}, b_{i2})$.
The \emph{reduced Gale configuration} $\tilde{B} = \{\tilde{b}_1, \ldots, \tilde{b}_n \}$
is obtained by setting 
\[
\tilde{b}_i =  \gcd(b_{i1}, b_{i2})^{-1} (-b_{i2}, b_{i1}).
\]
That is, $\tilde{B}$ is obtained from $B$ by rotating the vectors by $90$ degrees
and scaling so that elements in each vector are relatively prime.
For the notion of a minimal generating set to be meaningful, we
need to assume that the toric ideal $I_A$ is positively graded.  In
terms of the reduced Gale configuration, this means that there is no nonzero
vector $w \in \qq^2$ such that $w^T \tilde{b}_i > 0$ for all $i$. 
With this assumption, the vectors $\tilde{b}_i$ can be ordered in such a way
that each pair $\tilde{b}_i, \tilde{b}_{i+1}$ span a cone such that no
other $\tilde{b}_j$ lies in the interior of the cone (where  $\tilde{b}_{n+1} = \tilde{b}_{1}$).

For each cone ${\rm cone}( \tilde{b}_i, \tilde{b}_{i+1})$, let $H_i$ be its
Hilbert basis, which is the minimum generating set of the monoid 
${\rm cone}( \tilde{b}_i, \tilde{b}_{i+1}) \cap \zz^2$.  Define the 
\emph{Hilbert basis} of the reduced Gale configuration to be the set:
\[
\calh_A  =  \{ u \in \zz^2 :  \mbox{ both } u \mbox{ and } -u \mbox{ are in }
H_1 \cup H_2 \cup \cdots \cup H_n  \}.
\]

\begin{thm}\cite[Theorem 3.7]{Peeva1998} \label{thm:Peeva}
Let $A \in \zz^{(n-2) \times n}$ have rank $n-2$, and $B$ the Gale configuration.
A vector $u \in \zz^2$ is in $\calh_A$ if and only if $p^{(Bu)_+}- p^{(Bu)_-}$ is an
indispensable binomial of the toric ideal $I_A$.  Furthermore, the indispensable binomials
are a generating set for $I_A$, unless there are no indispensable binomials, in
which case $I_A$ is a complete intersection.
\end{thm}

Hilbert bases are complicated to compute for general cones, but in dimension $2$ there
is a particularly simple geometric description.

\begin{prop}\label{prop:hilb}
Let $a,b \in \zz^2$ and let $P = {\rm cone}(a,b)$.  The Hilbert basis
of $P$ consists of all lattice points in the polyhedron ${\rm conv}( (P \cap \zz^2) \setminus \{(0,0)\})$
that are visible from the origin.
\end{prop}

Combining Theorems \ref{thm:Lawrence} and \ref{thm:Peeva}, the Graver basis of
$A$ can also be characterized in terms of the reduced Gale configuration.

\begin{cor}\label{cor:graverchar}
Let $A \in \zz^{(n-2) \times n}$ have rank $n-2$, and $B$ the Gale configuration. 
Suppose that $\ker_\zz A \cap \nn^n  = \{0\}$. 
A vector $u \in \zz^2$ has either $u$ or ${}-u$ $ \in  H_1 \cup  \cdots \cup H_{n}$ if and only if 
$p^{(Bu)_+}- p^{(Bu)_-}$ is a primitive binomial of the toric ideal $I_A$.
\end{cor}

\begin{proof}
For the Gale configuration $B$ of $A$, define $B^\pm : =  B  \cup -B$,
which is the Gale configuration of the Lawrence lifting $\Lambda(A)$, and
let $\tilde{B}^\pm$ be its reduced Gale configuration.
As for $B$, we assume that the elements of $\tilde{B}^\pm$ are ordered so that
each cone ${\rm cone}( \tilde{b}_i, \tilde{b}_{i+1})$ no other $\tilde{b}_j$ lies
in its interior.  
Let $H^{\pm}_i$  be the Hilbert basis of ${\rm cone}( \tilde{b}_i, \tilde{b}_{i+1})$.
Since $\tilde{B}^\pm$ is centrally symmetric, the Hilbert basis of the
resulting Lawrence configuration $\Lambda(A)$ will be the union
of all the $H^{\pm}_i$.  By Theorem \ref{thm:Peeva} these vectors
determine the minimal generating set of $I_{\Lambda(A)}$.  By Theorem \ref{thm:Lawrence}
those vectors then determine the Graver basis of $I_A$.  So to
prove the corollary, we need to show that every $u$ in some $H^{\pm}_i$,
either $u$ or $-u$ appears in $H_1 \cup  \cdots \cup H_{n}$.

So let $u \in H^{\pm}_i$.  If $\tilde{b}_i$ and $\tilde{b}_{i+1}$ are both in $B$
or both in $-B$, then ${\rm cone}( \tilde{b}_i, \tilde{b}_{i+1})$ or ${\rm cone}( -\tilde{b}_i, -\tilde{b}_{i+1})$
is one of the cones described in the Hilbert basis of $A$, so $u$ or $-u$ belongs to
$ H_1 \cup  \cdots \cup H_{n}$.  This leaves the case that $\tilde{b}_i \in B$ and 
$\tilde{b}_{i+1} \in -B$ (the reverse situation follows from a symmetric argument).
Looking at the ordering on $B$, there will be a unique smallest $j$ such that
$\tilde{b}_j \in B$ and ${\rm cone} (\tilde{b}_i, \tilde{b}_j)$ forms one of the
cones for computing   $H_1 \cup  \cdots \cup H_{n}$.  Similarly, the there is a unique
largest $k$ such that $-\tilde{b}_k \in B$ and ${\rm cone} (-\tilde{b}_k, -\tilde{b}_{i+1} )$ forms one of the
cones for computing   $H_1 \cup  \cdots \cup H_{n}$.  These vectors are guaranteed to 
exist by the positive grading assumption that 
Clearly, we have that $\ker_\zz A \cap \nn^n  = \{0\}$.
\[
{\rm cone}( \tilde{b}_i, \tilde{b}_{i+1}) =  {\rm cone} (\tilde{b}_i, \tilde{b}_j) \cap 
{\rm cone} (\tilde{b}_k, \tilde{b}_{i+1} ).
\]
Furthermore,  if we let 
\[
P_{i,i+1} = {\rm conv}( {\rm cone}( \tilde{b}_i, \tilde{b}_{i+1})\cap \zz^2 \setminus \{(0,0)\} )
\]
and defined $P_{i,j}$ and $P_{k,i+1}$ similarly, then we have that
\[
P_{i,i+1}  =  P_{i,j} \cap P_{k,i+1}.
\]
Since each of $P_{i,i+1}$, $P_{i,j}$, and $P_{k,i+1}$ is the convex hull of lattice points,
the lattice points visible from the origin in $P_{i,i+1}$, will be either
a lattice point visible from the origin in $P_{i,j}$ or in $P_{k,i+1}$, or both.
In the case that  $u \in P_{i,i+1}$ is a lattice point visible from the origin with $u \in P_{i,j}$,
then $u \in H_1 \cup  \cdots \cup H_{n}$.  In the case that $u \in P_{i,i+1}$ is a 
lattice point visible from the origin with $u \in P_{k,i+1}$ then $-u \in H_1 \cup  \cdots \cup H_{n}$.
\end{proof}

Corollary \ref{cor:graverchar} then reduces the problem of characterizing
strongly robust toric ideals in codimension $2$ to the following problem.

\begin{pr}
For which rank $n-2$ matrices $A \in \zz^{(n-2) \times n}$ is the Hilbert basis $\calh_A$
equal to $H_1 \cup H_2 \cup \cdots \cup H_n$?
\end{pr}

The answer is contained in the following Lemma.

\begin{lemma}\label{lem:goodlemma}
Let  $A \in \zz^{(n-2) \times n}$ have rank $n-2$ and $\tilde{B}$ the reduced
Gale diagram.  Then $I_A$ is strongly robust if and only if for
each $b_i \in \tilde{B}$, $-b_i \in \calh_A$.
\end{lemma}

\begin{proof}
Clearly the condition of the theorem is necessary since $b_i$ itself always belongs to
$H_1 \cup \cdots \cup H_n$.  On the other hand, if $-b_i \in \calh_A$ but $-b_i \notin B$,
we can add it to $B$ without changing $\calh_A$.  Indeed, if ${\rm cone}(b_j, b_{j+1})$
contains $-b_i$ as a visible lattice point of $P_{j,j+1}$,
then the visible lattice points arising in the cones 
${\rm cone}(b_j, -b_i)$ and ${\rm cone}(-b_i, b_{j+1})$ are precisely the visible lattice 
points in $P_{j,j+1}$.  By repeating this procedure, we end up with a Gale diagram 
that contains only pairs $b_i, -b_i$, which is the Gale diagram of a Lawrence matrix.
Hence, $I_A$ is strongly robust.
\end{proof}

Now we are in a position to Prove Theorem \ref{thm:main}.

\begin{proof}[Proof of Theorem \ref{thm:main}]
Suppose that $I_A$ is a strongly robust toric ideal in codimension $2$, let $\tilde{B}$ be the
reduced Gale configuration, and $P = {\rm conv}(\tilde{B})$ the convex hull
of the elements in $\tilde{B}$.  Let $\tilde{b}$ be a vertex of $P$.
We must show that $-\tilde{b}$ is also a vertex of $P$ to see that $P$ is
centrally symmetric.

Since $\tilde{b} \in B$, and $I_A$ is strongly robust, $-\tilde{b}$ belongs to $\calh_A$,
by Lemma \ref{lem:goodlemma}.  If $-\tilde{b}$ is not a vertex of $P$, then there are
two vectors $b_1, b_2 \in \tilde{B}$ such that $-\tilde{b}$ is in ${\rm conv}(b1, b2, (0,0))$.
Applying Lemma \ref{lem:goodlemma} again, we have the $-b_1$ and $-b_2$ are in $\calh_A$.
In particular, these two vectors are in $P$, by Proposition \ref{prop:hilb}.  However, this
forces that $\tilde{b} \in {\rm conv}(-b1, -b2, (0,0))$, so $\tilde{b}$ could not be a
vertex of $P$.
\end{proof}

\begin{proof}[Proof of Corollary \ref{cor:2bouquet}]
Since the polytope $\conv(\tilde{B})$ must be two dimensional and is centrally symmetric,
it must have at least two pairs of opposite vertices $\tilde{b}_1, -\tilde{b}_1$ and $\tilde{b}_2, -\tilde{b}_2$.
These two pairs of opposite vertices yield two mixed bouquets of the matrix $A$.
\end{proof}

\begin{ex}\label{ex:Gale}
Let $A$ be the matrix
\[
A = \begin{pmatrix}
1 & 0 & 1 & 0 & 0 & 0 \\
0 & 1 & 0 & 0 & 1 & 0 \\
0 & 1 & 0 & 1 & 0 & 1 \\
-2 & 0 & 0 & 0 & -4 & 5
\end{pmatrix}
\]
which has the Gale transform $B$  and reduced Gale transform $\tilde{B}$ respectively:
\[
B = 
\begin{pmatrix}
1 & 2 \\
-2 & 1 \\
-1 & -2 \\
0 & -1 \\
2 & -1 \\
2 & 0 \\
\end{pmatrix} \quad \quad  \tilde {B} =
\begin{pmatrix}
-2 & 1 \\
-1 & -2 \\
2 & -1 \\
1 & 0  \\
1 & 2  \\
0 & 1
\end{pmatrix}
\]
As can be see from the reduced Gale transform, illustrated in Figure \ref{fig:Gale},
the matrix $A$ satisfies the condition of Lemma \ref{lem:goodlemma}, and so the
toric ideal is strongly robust.  The minimal generating set, which equals the Graver basis,
consists of the following $6$ binomials that are in bijection with pairs of opposite lattice
points in $\mathcal{H}_A$:
\[
I_A = \langle
b^{5}-d e^{5} f^{4},
a e^{2} f^{2}-b^{2} c,
a b^{3}-c d e^{3} f^{2},
a^{2} b-c^{2} d e,
a^{3} e f^{2}-b c^{3} d,
a^{5} f^{2}-c^{5} d^{2}
\rangle.
\]

For example, the binomial $a^{3} e f^{2}-b c^{3} d$ corresponds to the point $u = (1,1)$ in the Gale diagram,
since $Bu = (3, -1, -3, -1, 1,2)^T$.

\begin{figure}
\resizebox{!}{5cm}{\includegraphics{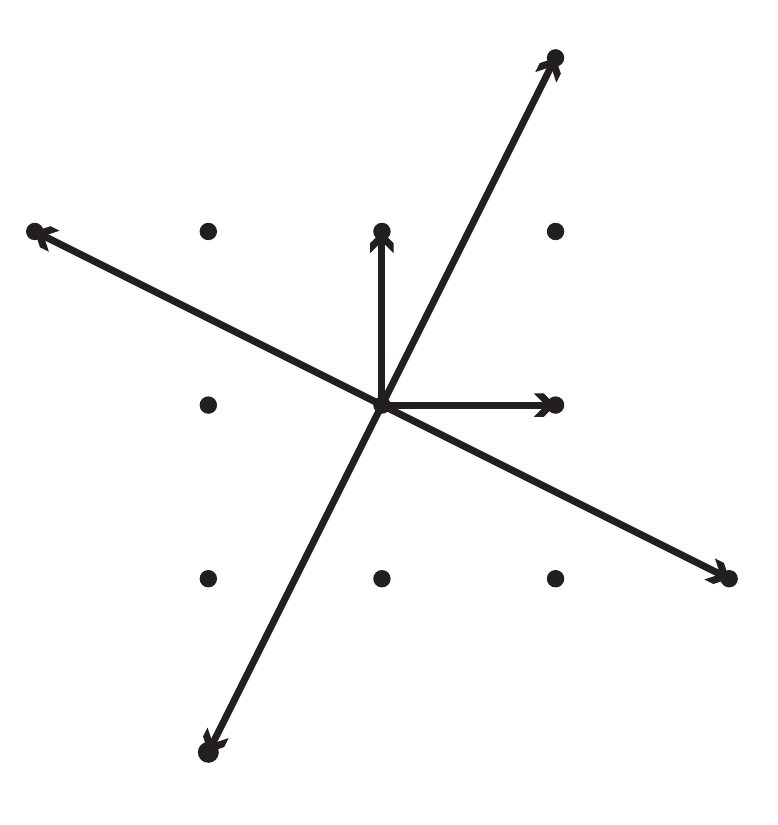}}
\caption{The reduced Gale transform of the matrix in Example \ref{ex:Gale}.  The dots represent the
points in the set $\mathcal{H}_A$. \label{fig:Gale}}
\end{figure}

\end{ex}

\section*{Acknowledgments}
Seth Sullivant was partially supported by the David and Lucille Packard 
Foundation and the US National Science Foundation (DMS 1615660).

\bibliographystyle{alpha}

\bibliography{robust}

\end{document}